\documentclass{amsart}
\usepackage[group-separator={,}, output-exponent-marker = \text{e}]{siunitx}
\usepackage[section]{placeins}
\usepackage{amssymb, graphicx,booktabs,mathtools, microtype, parskip}
\usepackage{pdflscape}

\usepackage{hyperref}
\usepackage{bookmark}

\newcount\hour \newcount\minute \newcount\minus
\hour=\time \divide \hour by 60
\minute=\time
\minus=\hour\multiply\minus by -60
\advance \minute by \minus
\def\now{%
\number\hour:%
  \ifnum \minute<10 0\fi%
  \number\minute%
}

\newtheorem{theorem}{Theorem}[section]

\newtheorem{proposition}[theorem]{Proposition}

\theoremstyle{definition}
\newtheorem{definition}[theorem]{Definition}

\theoremstyle{remark}
\newtheorem{remark}[theorem]{Remark}

\newcommand{\mat}[1]{\begin{bmatrix} #1 \end{bmatrix}}

\DeclarePairedDelimiterX\set[1]{\lbrace}{\rbrace}{\def\given{\;\delimsize\vert\;}#1}
\DeclarePairedDelimiter{\abs}{\lvert}{\rvert}
\DeclarePairedDelimiter{\ip}{\langle}{\rangle}

\newcommand{\cH}{\mathcal{H}}
\newcommand{\cF}{\mathcal{F}}
\newcommand{\ZZ}{\mathbb{Z}}
\newcommand{\QQ}{\mathbb{Q}}
\newcommand{\PP}{\mathbb{P}}
\newcommand{\RR}{\mathbb{R}}
\newcommand{\OO}{\mathcal{O}}
\newcommand{\CC}{\mathbb{C}}
\DeclareMathOperator{\GL}{GL}

\newcommand{\hthree}{\mathcal{H}^3(\mathbb{C})}
\DeclareMathOperator{\tr}{Tr}
\DeclareMathOperator{\Min}{Min}
\DeclareMathOperator{\spn}{span}

\DeclareMathOperator{\aff}{aff}

\DeclareMathOperator{\relint}{relint}
\DeclareMathOperator{\conv}{conv}

\DeclareMathOperator{\vertices}{vert}

\DeclarePairedDelimiter\lab{\langle}{\rangle}%
\DeclarePairedDelimiter\norm{\lVert}{\rVert}%

\begin{document}
\title{Ternary perfect forms over imaginary quadratic fields}
\author{Zachary Parker}
\address{UNCG, Greensboro, NC  27412}
\email{z\_parker@uncg.edu}
\urladdr{\url{https://sites.google.com/view/zacharyparker/welcome}}
\author{Dan Yasaki}
\address{UNCG, Greensboro, NC  27412}
\email{d\_yasaki@uncg.edu}
\urladdr{\url{https://go.uncg.edu/d_yasaki}}
\date{\today~\now}
%%11E20: General ternary and quaternary quadratic forms; forms of more than two variables
%%52C07: Lattices and convex bodies in $n$ dimensions
%%15A63: Quadratic and bilinear forms, inner products
%%11E39: Bilinear and Hermitian forms

\thanks{The first author thanks the University of North Carolina at Greensboro Graduate School for their support through a Summer Research Assistantship and the UNCG Dorothy Levis Munroe Student Research Fund. This work was supported by a grant from the Simons Foundation (848154, DY)
} 

\maketitle

%\begin{resume}
%Dans ce travail, nous \'{e}num\'{e}rons les formes ternaires parfaites sur tous les corps quadratiques imaginaires de discriminant absolu inf\'{e}rieur ou \'{e}gal \`{a} 91 et nous \'{e}tudions le nombre et les types combinatoires des polytopes associ\'{e}s. Nous d\'{e}montrons une borne sup\'{e}rieure au nombre de types combinatoires possibles, ind\'{e}pendamment du discriminant. Nous analysons les donn\`{e}es et formulons plusieurs observations concernant le nombre et la complexit\'{e} des formes parfaites qui apparaissent dans le cadre de nos calculs.
%\end{resume}

\begin{abstract}
In this work, we enumerate the ternary perfect forms over all imaginary quadratic fields of absolute discriminant up to $91$ and study the number and combinatorial types of their associated polytopes.  We prove a bound on the number of combinatorial types that can arise regardless of the discriminant.  We analyze the data and make several observations concerning the number and complexity of the perfect forms that arise in the scope of our calculation.
\end{abstract}

\section{Introduction}
Quadratic forms are prominent objects in mathematics. Historically, the classic geometric theory of these objects dates back to Minkowski's geometry of numbers \cite{MR249269}, contemporary with a problem of Hermite to find arithmetic minima of positive definite quadratic forms. Perfect forms over the rational numbers are positive definite real quadratic forms that are uniquely determined by their arithmetic minima and configuration of minimal vectors.  Their study goes back to the work on Korkin and Zolotorev \cite{korkine-zolotareff}.  Perfect forms and their corresponding lattices are important in the study of densest lattices packings and Hermite invariants \cite{martinet}.  Voronoi \cite{MR1580737} developed an algorithm for enumerating $n$-ary perfect forms and shows there are finitely many up to a natural equivalence by $\GL_n(\ZZ)$.  They have been classified for $n\leq 8$.\footnote{M.~Dutour and W.~van Woerden have announced a classification for $n=9$ in a recent preprint \cite{dutour9}.}  See \href{https://jamartin.perso.math.cnrs.fr/Lattices/index.html}{A catalogue of Perfect Lattices} written by Batut and Martinet and \href{https://www.math.rwth-aachen.de/~Gabriele.Nebe/LATTICES/index.html}{A Catalogue of Lattices} written by Nebe and Sloane for the perfect lattices and references.   See \cite{MR2466406, watanabe-survey} for additional history and details. 

Generalizations of Voronoi's algorithm over arbitrary number fields were developed by Ash \cite{MR457437} and Koecher \cite{MR124527}. There exists an infinite polyhedron $\Pi$ in the space of forms whose facets are parameterized by perfect forms. %, forms uniquely determined by their arithmetic minimum and minimal vectors.
In the classical case of binary perfect forms over $\QQ$, $\Pi$ descends modulo scaling to the Farey tessellation of the complex upper half-plane.  See Figure~\ref{fig:farey}.  The ideal hyperbolic triangle with vertices $0$, $1$, and $\infty$ corresponds to the normalized perfect form $\phi(x,y) = x^2 - xy + y^2$ that attains arithmetic minimum $1$ at $\pm (0,1)$, $\pm (1,1)$, and $\pm (1,0)$.
\begin{figure}
\includegraphics[width = 0.7\textwidth, trim={0.5cm 0 0.5cm 1cm},clip]{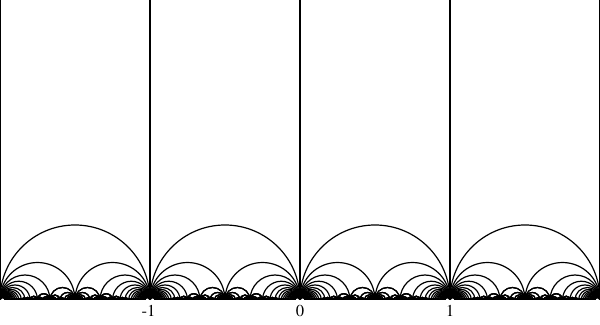}
  \caption{Farey tessellation of complex upper half-plane. The points $\frac{a}{c}$ and  $\frac{b}{d} \in \PP^1(\QQ)$ are joined by a hyperbolic geodesic when $\abs{ad - bc} = 1$, where we take $\infty  = \frac{1}{0}$.} \label{fig:farey}
\end{figure}

For a number field $F$ with ring of integers $\OO_F$, there are finitely many perfect forms over $F$ up to equivalence by $\GL_n(\OO_F)$.  These have been enumerated for small values of $n$ and small degree number fields $F$.  In this generality, even the $n=1$ case is not completely understood outside of explicit calculations for small degree fields.  See \cite{watanabe-tr, watanabe-unary,Yasunary,Yascyclotomic,imquad-perfect,sigrist} for a sampling of the known calculations. 
There are related but different notions of perfect forms that have also been studied \cite{MR3315516, coulangeon-voronoi, bci-quad}.

In this paper, we enumerate the $\GL_3(\OO_F)$-equivalence classes of ternary perfect forms over imaginary quadratic fields of discriminant $D\geq -91$.  We introduce a coarser combinatorial equivalence on these forms and enumerate the combinatorial classes.  We investigate how these notions of equivalence vary as the the discriminant and class number vary and make several observations.  Scheckelhoff, Thalagoda, and the second author examined similar questions in the binary case in \cite{imquad-perfect}. 

The explicit calculation of these perfect forms has immediate applications to the calculation of group cohomology \cite{aim-coh} and more generally to the study of certain spaces of automorphic forms \cite{MR1603257}.  This has been successful in explicit explorations of these spaces.  See \cite{gl3neg3} and the papers referenced there.

The paper reads as follows: In Section~\ref{sec:perfect}, we set notation recall the theory of perfect forms and their associated polytopes.  In Section~\ref{sec:gale}, we describe the Gale-diagram, a useful tool for classifying the combinatorial types of polytopes with few vertices. This proved vital in classifying the combinatorial types of the perfect forms.  An overwhelming  majority of the polytopes had few vertices and without an efficient method of determining the types, the computation would have been infeasible.  In Section~\ref{sec:data}, we present the data and make some observations.  We provide several explicit examples. 

%The first author thanks the University of North Carolina at Greensboro Graduate School for their support through a Summer Research Assistantship.  This work was supported by a grant from the Simons Foundation (848154, DY).

\section{Perfect forms and Polytopes}\label{sec:perfect}
This section recalls some definitions to accustom the reader to our notation. We follow \cite{imquad-perfect} and \cite{aim-coh} closely, and the reader should refer to those resources for a more detailed account.
\subsection{Perfect forms}
Let $d$ be a fixed, squarefree, positive integer. Let $F$ be the imaginary quadratic field $F=\mathbb{Q}(\sqrt{-d})$ with corresponding ring of integers $\OO_F$. If $d\equiv 1,2 \bmod{4}$, then $F$ has discriminant $\Delta=-4d$, and if $d\equiv 3\bmod{4}$, then $F$ has discriminant $\Delta=-d$. The ring of integers $\OO_F$ is equal to $\mathbb{Z}[\omega]$ where
    \[
    \omega= \begin{cases}
        \sqrt{-d} & \text{if $d\equiv 1,2\bmod{4}$,} \\
        (1+\sqrt{-d})/2 & \text{if $d\equiv 3\bmod{4}$.}
    \end{cases}
    \]
Fix a complex embedding $F \hookrightarrow \mathbb{C}$, and identify $F$ with its image. We extend this identification to vectors and matrices. We will use $\overline{\cdot}$ to denote complex conjugation, the nontrivial Galois automorphism of $F$. 

Let $\hthree$ denote the $9$-dimensional real vector space of $3\times 3$ Hermitian matrices with complex coefficients.
    Let $C\subset \hthree$ denote the subset of positive definite matrices.  Then $C$ is a codimension $0$ open cone whose boundary consists of positive semidefinite Hermitian matrices. Via the fixed complex embedding of $F$, view $\mathcal{H}^3(F)$, the $3\times 3$ Hermitian matrices with entries in $F$, as a subset. Define $q\colon \OO_F^3  \rightarrow \cH^3(F)$ by $q(v)=vv^*$, where $v^*$ denotes the conjugate transpose.
    
Each $A\in\hthree$ defines a Hermitian form on $\mathbb{C}^3$ denoted $A[x]$,  where
    \[
    A[x] = xAx^*, \quad x\in \CC^3.
    \]
Define the nondegenerate bilinear pairing
    \[
    \ip{\cdot, \cdot} \colon \hthree \times \hthree \rightarrow \mathbb{C}
    \]
via $\ip{A,B} = \tr(AB)$.  When we restrict the domain of the Hermitian form $A$ to $\OO_F^3$, we call $A$ a \emph{Hermitian form over $F$}.  For $v\in \OO_F^3$, properties of trace imply that
    \[
    A[v]=\tr(A(q(v))=\ip{A,q(v)}.
    \]
    It follows that specifying the value of a Hermitian form at $v \in \OO_F^3$ is a linear condition on $\cH^3(\CC)$.

    \begin{definition}
    For $A\in C$, define the \emph{minimum of $A$} as
        \[
        \min(A) = \inf_{v\in \OO_F^3 \setminus \set{0}} A[v].
        \]
        A vector $v\in\OO_F^3$ is called a \emph{minimal vector of $A$} if $A[v]=\min(A)$ and $\Min(A)$ will denote the set of minimal vectors of $A$.
\end{definition}
Note that for a unit $\lambda \in \OO_F^\times$ and a vector $v \in \OO_F^3$, we have $A[v] = A[\lambda v]$ since $\lambda \overline{\lambda} = 1$.  In particular, if $v \in \Min(A)$ then $\lambda v \in \Min(A)$ for every $\lambda \in \OO_F^\times$.  Thus we follow the convention that $\Min(A)$ is a set of minimal vectors up to units.  In particular, a form $A$ actually has $(\#\Min(A))(\#\OO_F^\times)$ minimal vectors.
    
Note that $\min(A)>0$ since $A$ is positive definite. Since $\OO_F^3$ is discrete in the topology of $\mathbb{C}^3$, the set $\Min(A)$ is finite. The following theorem gives a bound on the number of minimal vectors for a positive definite ternary Hermitian form over any imaginary quadratic field.

\begin{theorem} \label{thm:minvec-bound}
    Let $A$ be a positive definite ternary Hermitian form over an imaginary quadratic field. Then
    \[
    \#\Min(A)\leq 36.
    \]
\end{theorem}

\begin{proof}
    Fixing a $\ZZ$-basis $\set{1,\omega}$ for $\OO_F$, we have a bijection $\phi \colon \OO_F^3\rightarrow \mathbb{Z}^6$. This allows us to take $3\times 3$ Hermitian matrices to $6\times 6$ symmetric matrices. Hence, there is an induced map $\Phi\colon V_F \rightarrow V_\mathbb{Q}$, where $V_F$ represents positive definite Hermitian forms over $F$ to positive definite senary forms over $\mathbb{Q}$. Further, $\Phi$ preserves vector evaluation. That is,
    \[
    A\left[v\right]=\Phi(A)\left[\phi(v)\right]
    \]
    for all $v\in\OO_F^3$. Thus, the value of $\min(\Phi(A))$ is wholly determined by the value of $\min(A)$, and so minimal vectors are preserved. Further, this preservation ensures that $A$ cannot have a greater number of minimal vectors than $\Phi(A)$. Barnes \cite{MR86834} examines the seven perfect forms in the space of positive definite senary forms over $\mathbb{Q}$. Of these seven, there are three with $21$ (pairs) of minimal vectors, one with $22$, one with $27$, one with $30$, and one with $36$. Then every positive definite senary form over $\mathbb{Q}$ has at most $36$ minimal vectors. Hence, $A$ can have at most $36$ minimal vectors.
\end{proof}

\begin{definition}
    A positive definite Hermitian form $A\in C$ is \emph{perfect} (or \emph{a perfect Hermitian form over $F$}) if
    \[
    \spn_\mathbb{R}\set{q(v) \given v\in \Min(A)} = \hthree.
    \]
\end{definition}

Since $A[v]=\langle A,q(v)\rangle$, a vector $v\in \OO_F^3$ gives a linear functional on $\hthree$ defined by
\[A\mapsto A[v] = \ip{A,q(v)}.\]
It follows that a form is perfect if and only if it is uniquely determined by its minimum and its set of minimal vectors. It is clear that for a real number $\lambda > 0$ and a perfect form $A$, the form $\lambda A$ is also perfect with the same minimal vectors as $A$.  Thus by convention, we normalize each perfect form $A$ so that $\min(A)=1$. From work of Okuda-Yano \cite[Theorem~3.2]{Okuda2010AGO}, with this normalization a perfect form $A$ realizable over $F$, i.e., $A \in \cH^3(F)$.

\begin{definition}
Two perfect forms $A, B \in C$ are \emph{equivalent} if there exists $\gamma \in \GL_3(\OO_F)$ such that $A = \gamma B \gamma^*$.
\end{definition}
If $A$ and $B$ are equivalent perfect forms, then by our convention
\[\min(A) = \min(B) = 1 \quad \text{and} \quad \gamma^*\Min(A) =  \Min(B).\]

\subsection{Polytopes associated to perfect forms}

\begin{definition}
Let $V=\set{v_1,v_2,\dots, v_n}$ be a set of $n$ vectors in $\mathbb{R}^N$.  Let $\aff(V)$ be the \emph{affine subspace of all affine combinations of $V$},
\[ \aff(V) = \set*{\sum_{i = 1}^n c_iv_i \in \RR^N \given c_i \in \RR \quad \text{and} \quad  \sum_{i = 1}^n c_i = 1}.\]
The \emph{convex hull of $V$}, denoted $\conv(V)$, is the subset of $\aff(V)$ where the all of the $c_i$ are nonnegative.
\end{definition}

\begin{definition}
  Let $A \in C$ be a perfect Hermitian form over $F$.  Let $q(\Min(A)) \subset \cH^3(\CC)$ be the subset
  \[q(\Min(A)) = \set{q(v) \in \cH^3(\CC) \given v \in \Min(A)}.\]  The \emph{polytope associated to $A$}, denoted $P_A$, is the convex hull
  \[P_A = \conv(q(\Min(A))).\]
\end{definition}

Note that for a unit $\lambda \in \OO_F^\times$ and a vector $v \in \OO_F^3$, we have $q(v) = q(\lambda v)$ so with our convention that $\Min(A)$ is the set of minimal vectors of $A$ up to units, we have 
\[\#\vertices(P_A) = \#\Min(A).\]

For a polytope $Q$, let $\cF(Q)$ denote the faces of $Q$.  
\begin{definition}
  Two polytopes $P$ and $P'$ are \emph{combinatorially equivalent} if there exists an inclusion preserving bijection $\phi \colon \cF(P)\to \cF(P')$.
\end{definition}

Let $A$ and $B$ be perfect forms.  It is clear that if $A$ and $B$ are equivalent, then their associated polytopes $P_A$ and $P_B$ are combinatorially equivalent.  The converse is not true in general.  This leads to a coarser equivalence relation on the set of perfect forms.
\begin{definition}
  Let $A$ and $B$ be perfect forms.  We say $A$ and $B$ are \emph{combinatorially equivalent} if their associated polytopes $P_A$ and $P_B$ are combinatorially equivalent.
\end{definition}

\begin{theorem}
There are only finitely many combinatorial types of ternary perfect forms over imaginary quadratic fields.
\end{theorem}
\begin{proof}
Let $A$ be a ternary perfect form.  From Theorem~\ref{thm:minvec-bound}, we have $\#\Min(A) \leq 36$.  Thus the polytope $P_A$ associated to $A$ has at most $36$ vertices.  There are only finitely many combinatorial types of polytopes with at most $36$ vertices, so the result follows.
\end{proof}

\begin{remark}
There is a similar result bounding the number of combinatorial types of binary perfect forms over imaginary quadratic fields \cite[Remark~4.2]{imquad-perfect}.  In that case, of the \num{6860405} combinatorial types of
$3$-dimensional polytopes with at most $12$ vertices, only $8$ combinatorial types were observed.  There is a far greater number of combinatorial types of $8$-dimensional polytopes with at most 36 vertices.  Indeed, for higher dimensional polytopes complete enumeration is still open, but bounds are known \cite{many-polytopes}.  While there is a vast number of combinatorial types,  again only a small proportion was observed in our data.  See Section~\ref{sec:data} for details.
\end{remark}

\section{Gale-transforms, Gale-diagrams, and Gale-labels}\label{sec:gale}

Gale-transforms and their associated Gale-diagrams are useful for classifying combinatorial equivalence classes of polytopes.  The theory is best suited for handling polytopes with few vertices relative to the dimension of the polytope.  We provide just enough details to describe the technique and set notation.  For a more detailed description, see \cite[Chapters~5--6]{grunbook}.
\subsection{Gale-transforms}
\begin{definition}
Let $V=\set{v_1,v_2,\dots, v_n}$ be a set of distinct vectors in $\mathbb{R}^N$.
  The set of \emph{affine dependencies of $V$}, denoted $D(V)\subseteq \mathbb{R}^n$, is the collection
\[D(V) = \set*{(a_1,a_2, \dots,a_n) \in \mathbb{R}^n \given
    \sum_{i=1}^n a_i v_i = 0 \quad \text{and} \quad
    \sum_{i=1}^n a_i =0}.
\]
\end{definition}

\begin{definition}
  Let $D(V)$ be the set of affine dependencies of a set of vectors $V$.  Let $D_1$ be a matrix whose columns form an ordered basis of $D(V)$.  The \emph{Gale-transform of $V$}, denoted $\overline{V}$, is the sequence of rows of $D_1$.
\end{definition}

Let $V=\set{v_1,v_2,\dots, v_n}$ be a set of vectors in $\mathbb{R}^N$.
Suppose $\dim(\aff(V)) =d$.  Let $\delta$ denote the difference,
\[\delta = n - (d + 1).\]  Then $D(V)$ is an $\delta$-dimensional subspace of $\RR^n$.  The Gale-transform assigns a point $\overline{v}_j\in\mathbb{R}^{\delta}$ to each point $v_j\in V$.
Note that the $n$-tuple $\overline{V}$ need not consist of $n$ distinct points.  Thus we  view $\overline{V}$ as a multiset.  

\begin{remark}
Although Gale-transforms can be defined generally, we use them to study convex polytopes. Let $P$ be a $d$-dimensional polytope with $n$ vertices embedded in $\RR^N$.  In practice, we choose $N = d + 1$.  The combinatorial type of $P$ is determined by the Gale-transform of the vertex set of $P$.  Note that in this setting, if $n = d + 1$ then $P$ is a simplex.  Thus $\delta$ is a crude measure of how far from a simplex $P$ is.  In this way, the theory of Gale-transforms turns the study of $d$-dimensional polytopes in $\RR^N$ into the study of certain multisets of $\RR^{\delta}$.
\end{remark}

\begin{definition}
  Let $B$ denote the unit ball in $\RR^m$, and let $S \subset \RR^m$ be a convex set.  The \emph{relative interior of $S$}, denoted $\relint(S)$ is
  \[\relint(S) = \set{s \in \RR^m \given \aff(S) \cap (s + \epsilon B) \subset S \quad \text{for some $\epsilon > 0$.}}\]
\end{definition}

\begin{definition}\label{def:gale-iso}
  Let $V = \set{v_1, v_2, \dots,v_n }$ and $V' = \set{v'_1, v'_2, \dots,v'_n }$ be ordered sets of vectors.  The Gale-transforms $\overline{V}$ and $\overline{V}'$ are \emph{isomorphic} if there exists a permutation $\sigma \in S_n$ such that for every  $J\subseteq\{1,2,\dots,n\}$, the condition $0\in\relint(\conv (\overline{V}(J)))$ is equivalent to $0\in \relint(\conv(\overline{V}'(\sigma(J))))$, where $\sigma(J)=\set{\sigma(j) \given j\in J}$.  
\end{definition}

\begin{theorem}[{\cite[Theorem 5.4.5]{grunbook}}] \label{thm:gale}
Let $P$ and $P'$ be $d$-dimensional polytopes.  Then $P$ and $P'$ are combinatorially equivalent if and only if the Gale-transforms of their vertex sets are isomorphic.
\end{theorem}

\subsection{Gale-diagrams}
Note that scaling each point of a Gale-transform by a positive scalar does not change its isomorphism class.    In particular, we can arrange that all of the points of a Gale-transform lie on a sphere.  This gives rise to the notion of a Gale-diagram.
\begin{definition} \label{defn:gale-diagram}
For any Gale-transform $\overline{V}$ of $V$, define the \emph{Gale-diagram} $\hat{V}$ of $V$ to be the sequence of vectors $\hat{V}=(\hat{v}_1,\hat{v}_2,\dots,\hat{v}_n)$, where
\[
    \hat{v}_i = \begin{cases}
    0 & \text{if $\overline{v}_i=0$,} \\
    \frac{\overline{v}_i}{\norm{\overline{v}_i}} & \text{if $\overline{v}_i\neq 0$.}
    \end{cases}
\]
Here $\norm{x}$ denotes the Euclidean length of $x$.  Two Gale-diagrams are \emph{isomorphic} if the associated Gale-transforms are isomorphic.
\end{definition}

The Gale-transform of the vertex set of a $d$-dimensional polytope with $n$ vertices is a subset of $\set{0}\cup S^{\delta - 1}$, where $\delta = n - (d + 1)$ and $S^{k}$ is the unit $k$-sphere centered at the origin.
As before with Gale-transforms, the points in a Gale-diagram are not necessarily distinct so we view them as multisets of points in $\set{0}\cup S^{\delta-1}$.  

We can rotate any of the points of a Gale-diagram and stay in the same isomorphism class provided the relative interior condition of Definition~\ref{def:gale-iso} is maintained.  In particular, we can rotate points so that they coincide as much as possible within an isomorphism class.  A Gale-diagram is said to be \emph{contracted} if it has the minimal number of distinct points within its isomorphism class. 

\subsection{Gale-labels}
The theory of Gale-diagrams is simplest when $\delta$ is small.  When $\delta = 0$, $P$ is a simplex.  We focus on the cases $\delta = 1$ and $\delta = 2$.  In these cases, we can explicitly describe representatives for the isomorphism classes of contracted Gale-diagrams.  We introduce a notion of Gale-label, which allows us to implement an algorithm to effectively compute and store Gale-diagrams for polytopes with few vertices.

Let $P$ be a $d$-dimensional polytope with $d+2$ vertices $V=\vertices(P)$.  Then $\delta = 1$. The Gale-transform $\overline{V}$ is a $(d+2)$-tuple of points in $\mathbb{R}$. The Gale-diagram $\hat{V}$ is contained in the 3-point set $\set{0}\cup S^{0} = \set{0,\pm 1} \subseteq \mathbb{R}$.  Let $m_i$ denote the multiplicity of $i \in \overline{V}$.  The multiplicities are nonnegative and sum to $d + 2$.  Furthermore, the multiplicities of $1$ and $-1$ are greater than or equal to $2$ and can be interchanged to yield an isomorphic Gale-diagram.  This fixes a unique representative in each isomorphism class of such Gale-diagrams.

\begin{definition} 
A \emph{Gale-label on $d+2$} is a triple of nonnegative integers $[a_1,b,a_2]$ satisfying
\begin{enumerate}
\item $a_1+b+a_2=d+2$
\item $a_1 \geq a_2\geq 2$
\end{enumerate}  
\end{definition}

Since combinatorial equivalence classes of $d$-dimensional polytopes are parameterized by isomorphism classes of Gale-transforms by Theorem~\ref{thm:gale}, we get a unique label for each class. 
\begin{theorem}\label{thm:d+2}
    There is a bijection between combinatorial equivalence classes of $d$-dimensional polytopes with $d+2$ vertices and Gale-labels on $d+2$.
\end{theorem}

Let $P$ be a $d$-dimensional polytope with $d+3$ vertices. Then $\delta = 2$ so the Gale-diagrams is contained in $\set{0} \cup S^1$.  We visualize the Gale-diagram $\hat{V}$ of $V=\vertices(P)$ as $S^1$ along with all diameters of $\hat{V}$; that is, diameters of $S^1$ which have at least one endpoint in $\hat{V}$. We label each endpoint and the origin with its multiplicity in $\hat{V}$. Note that if an endpoint is not in $\hat{V}$, the multiplicity is $0$.  By construction, there is no diameter with both endpoints having multiplicity $0$.  The isomorphism class of $\hat{V}$
has a representative for which the consecutive diameters of its Gale-diagram are equidistant. Such Gale-diagrams are called \emph{standard}.  See Figure~\ref{fig:gale-diagram}.

\begin{figure}
\includegraphics[width=0.8\textwidth]{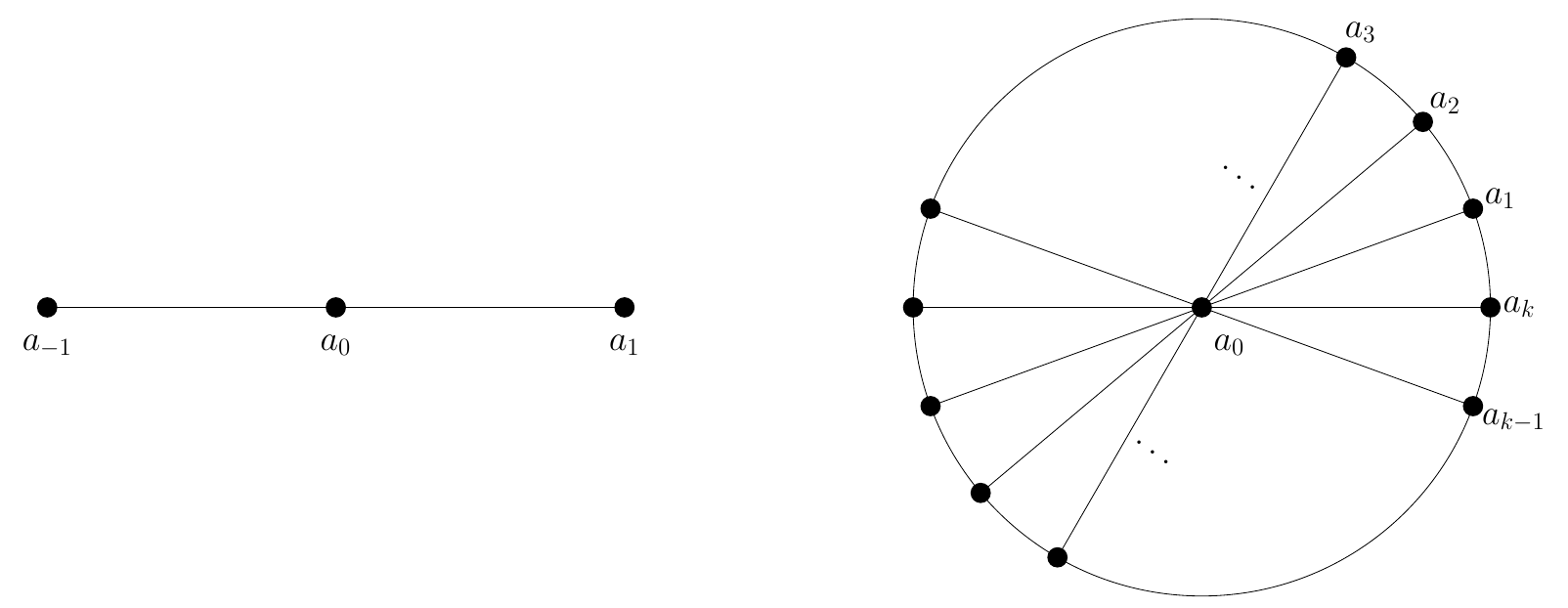}  
  \caption{\textbf{Left:} A visualization of the contracted standard Gale-diagram $\hat{V}$ of a vertex set $V$ of a $d$-dimensional polytope with $d+2$ vertices. The multiplicities are marked on $a_1, b$, and $a_2$. \textbf{Right:} A visualization of the contracted standard Gale-diagram $\hat{V}$ of a vertex set $V$ of a $d$-dimensional polytope with $d + 3$ vertices. The ends of the diameters and the origin are marked with the multiplicities $a_i$ of the point in $\hat{V}$.}
  \label{fig:gale-diagram}
\end{figure}

\begin{theorem}[{\cite[Theorem 6.3.1]{grunbook}}]\label{thm:gale-orthogonal}
  Two $d$-dimensional polytopes with $d+3$ vertices are combinatorially equivalent if and only if the contracted standard forms of their Gale-diagrams are orthogonally equivalent; that is, isomorphic under an orthogonal linear transformation of $\mathbb{R}^2$ onto itself.
\end{theorem}

We now extend the definition of Gale-labels to cover $d$-dimensional polytopes with $d+3$ vertices.
\begin{definition}
A \emph{Gale-label on $d+3$} is a sequence of nonnegative integers $a_0, a_1,\dots, a_k$  with $k$ even such that 
    \begin{enumerate}
        \item $\displaystyle\sum_{i=0}^k a_i = d+3$; \label{it:sum}
        \item $a_1\geq a_i$ for all $2\leq i\leq k$; and  \label{it:max1}
        \item if $a_i=a_1$ for some $2\leq i\leq k$, then in lexicographic order we have \label{it:max2} 
          \begin{enumerate}
          \item $[a_1,a_2,\ldots,a_k]>[a_j,a_{j+1},\ldots,a_k,a_1,\ldots,a_{j-1}]$ and
          \item $[a_1,a_2,\ldots,a_k]>[a_j,a_{j-1},\ldots,a_1,a_k,\ldots,a_{j+1}]$.
          \end{enumerate}
    \end{enumerate}
We notate the Gale-label as $\lab{ a_0, [ a_1,a_2,\ldots, a_k]}$.
\end{definition}

\begin{theorem}
    There is an injective map from the set of combinatorial equivalence classes of $d$-dimensional polytopes with $d+3$ vertices to Gale-labels on $d+3$.
\end{theorem}

\begin{proof}
  First we construct the map.  From the construction, it will be clear that the map is an injection.
  
  Let $P$ be a $d$-dimensional polytopes with $d+3$ vertices.  Let $\hat{V}$ denote its contracted standard Gale-diagram.  Recall that $\hat{V}$ can be visualized as $\set{0}\cup S^1$ with equally spaced diameters as depicted in Figure~\ref{fig:gale-diagram}.  Then endpoints of the diameters and the origin are labeled with multiplicities $a_0, a_1, \dots, a_k$.  We use the multiplicities to create the label $\lab{ a_0, [ a_1,a_2,\ldots, a_k]}$. It remains to show this label satisfies the conditions to be a Gale-label.

  The $a_i$ are nonnegative by construction.  The multiplicities sum to the number of vertices of $P$, so condition \eqref{it:sum} is satisfied.  

  By Theorem~\ref{thm:gale-orthogonal}, two contracted standard Gale-diagrams are isomorphic if one is obtained from the other via an orthogonal transformation.  Thus we can modify the diagram by rotations about the origin and reflections about lines through the origin and stay in the same isomorphism class.  On the label, these actions preserve $a_0$.  The actions on $[a_1, a_2, \dots, a_k]$ are generated by a cyclic shift $[a_1, a_2, \dots, a_k] \mapsto [a_2, a_3, \dots, a_k, a_1]$ and reversal $[a_1, a_2, \dots, a_k] \mapsto [a_k, a_{k-1}, \dots, a_1]$.  Among the at most $2k$ labels that can be obtained using these actions, choose the one that has $[a_1,a_2,\dots,a_k]$ maximal with respect to lexicographic ordering.  This ensures 
  conditions \eqref{it:max1} and \eqref{it:max2} from the definition.  Thus the map $P \mapsto \lab{ a_0, [ a_1,a_2,\ldots, a_k]}$ gives the desired map.
\end{proof}

Recall that a polytope $P$ is \emph{simplicial} if each facet of $P$ is a simplex.  This property can be read off from the Gale-label of the polytope.
\begin{proposition}[{\cite[Theorem~6.1.1]{grunbook}}]\label{prop:simplicial}
  A $d$-dimensional polytope with $d+2$ vertices and Gale-label $[a_1, b, a_2]$ is simplicial if and only if $b=0$.   
\end{proposition}

\section{Computational results and observations}\label{sec:data}
In this project, we investigate imaginary quadratic fields of absolute discriminant less than or equal to $91$.  This is $30$ fields.  For each field $F$, we enumerate the $\GL_3(\OO_F)$-equivalence classes of ternary perfect forms over $F$ as well as the combinatorial equivalence classes.  The calculation of the combinatorial equivalence classes would have been infeasible without the theory of Gale diagrams to provide efficient equivalence checking for the vast majority of perfect forms in the data.

\subsection{Simplex forms}
Since $\dim(\cH^3(\CC)) = 9$, the fewest minimal vectors a perfect form could have is $9$.  Such forms occur for most of the imaginary quadratic fields in the scope of the computation, and it appears to be quite common.  In fact, the data suggest that as the absolute discriminant of $F$ grows, the proportion of perfect forms over $F$ with exactly 9 minimal vectors approaches 100\%.  See Table~\ref{tab:simplex-count}.  Up to combinatorial equivalence, the simplex is the only $8$-dimensional polytope with exactly $9$ vertices.  Thus the configuration of minimal vectors for such a perfect form is a simplex.  Scheckelhoff, Thalagoda, and the second author noticed similar phenomena for perfect binary forms over $F$ in \cite[Figure 3]{imquad-perfect}, \cite[Tables 3--4]{Yasbianchi}, where the most common configuration formed a tetrahedron. In light of this, we call a perfect ternary Hermitian form a \emph{simplex form} if it has exactly $9$ minimal vectors.  Table~\ref{tab:simplex-count} summarizes the results for simplex forms in the range of computation.
From Table~\ref{tab:counts}, we see that $\num{2055920}$ (or $94.10\%$) of the \num{2184773} perfect forms observed are simplex forms.

\begin{table} 
  \caption{The number $N$ of equivalence classes of simplex forms and the proportion $P$ of those simplex forms among all equivalence classes of perfect forms over the imaginary quadratic field of discriminant $\Delta$.} \label{tab:simplex-count}
  $
\begin{array}{c@{\hspace{25pt}}c}
\begin{array}{rrS[round-mode=places, round-precision=2]}
\toprule
\Delta & N & {P}  \\
\midrule
-3 & \num{ 1 } & 50.00 \\
-4 & \num{ 0 } & 0.0000 \\
-7 & \num{ 0 } & 0.0000 \\
-8 & \num{ 0 } & 0.0000 \\
-11 & \num{ 3 } & 25.00 \\
-15 & \num{ 53 } & 58.89 \\
-19 & \num{ 115 } & 73.25 \\
-20 & \num{ 149 } & 70.28 \\
-23 & \num{ 642 } & 73.79 \\
-24 & \num{ 378 } & 63.42 \\
-31 & \num{ 3351 } & 84.77 \\
-35 & \num{ 3185 } & 87.14 \\
-39 & \num{ 9991 } & 84.91 \\
-40 & \num{ 8299 } & 89.44 \\
-43 & \num{ 7720 } & 91.28 \\
\bottomrule
\end{array}
& \begin{array}{rrrS[round-mode=places, round-precision=2]}
\toprule
\Delta & N & {P}  \\
\midrule
-47 & \num{ 22364 } & 81.21 \\
-51 & \num{ 16590 } & 87.27 \\
-52 & \num{ 25490 } & 90.90 \\
-55 & \num{ 49019 } & 92.30 \\
-56 & \num{ 44869 } & 92.01 \\
-59 & \num{ 41945 } & 93.36 \\
-67 & \num{ 55076 } & 95.83 \\
-68 & \num{ 106494 } & 94.46 \\
-71 & \num{ 185538 } & 93.73 \\
-79 & \num{ 246804 } & 95.45 \\
-83 & \num{ 168583 } & 92.95 \\
-84 & \num{ 233432 } & 94.07 \\
-87 & \num{ 371424 } & 93.30 \\
-88 & \num{ 247277 } & 96.78 \\
-91 & \num{ 207128 } & 96.59 \\
\bottomrule
\end{array}
\end{array}
  $  
\end{table}
\subsection{10 minimal vectors}
As we see from Table~\ref{tab:counts} (why is table 6 so far away? should we move it?), perfect forms with $10$ minimal vectors are the second most common in the data.  From Theorem~\ref{thm:d+2}, we see that there are several combinatorial types of convex polytopes with $10$ vertices.  Each combinatorial type can be determined uniquely by its Gale-label $[a_1,b,a_2]$.  It is straightforward to see that there are $16$ combinatorial types given by labels,
\begin{multline*}
  \{[2,6,2],[3,4,3],[3,5,2],[4,2,4],[4,3,3],[4,4,2],[5,0,5],[5,1,4],\\
    [5,2,3],[5,3,2],[6,0,4],[6,1,3],[6,2,2],[7,0,3],[7,1,2],[8,0,2]\}.
\end{multline*}
We observe $11$ of these types.  Our computations did not produce any perfect forms of combinatorial type $[4,4,2]$, $[5,3,2]$, $[6,2,2]$, $[7,1,2]$, or $[8,0,2]$.

From Table~\ref{tab:n10-count}, we see that the combinatorial types perfect forms are not evenly distributed.  Some types, such as $[3,4,3]$, arise in $26$ of the $30$ fields within the scope of the computation and account for $21.79\%$ of all perfect forms with $10$ minimal vectors.  Other types are more rare.  For example, combinatorial type $[7,0,3]$ occurs in only $16$ of the fields and accounts for $0.46\%$ of perfect forms with $10$ minimal vectors.  

We call a perfect form a \emph{simplicial form} if its polytope is simplicial.  By Proposition~\ref{prop:simplicial}, the Gale-label of a simplicial polytope with $10$ vertices has the form $[a_1,0,a_2]$, where $a_1 + a_2 = 10$ and $a_1 \geq a_2$.  One easily checks that there are $4$ combinatorial types of simplicial forms with exactly $10$ minimal vectors,
\[\set{[8,0,2], [7,0,3], [6,0,4],[5,0,5]}.\] Of these $4$ types, we observe $3$ in the data.  We do not have an example of $[8,0,2]$ type.  Given how rare the $[7,0,3]$ type is, it is plausible that we just have not computed out far enough to see any of type $[8,0,2]$.

\begin{figure}
  \includegraphics[width = 0.8\textwidth]{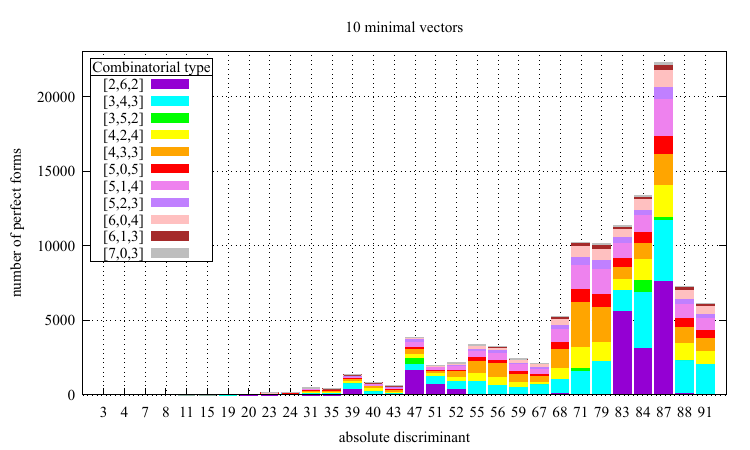}
  \caption{The distribution among the different combinatorial equivalence classes of the equivalence classes of perfect forms over the imaginary quadratic field of absolute discriminant $\abs{\Delta}$ with exactly $10$ minimal vectors.}
  \label{fig:n10}
\end{figure}

\begin{landscape}
 \begin{table}
  \caption{The distribution among the different combinatorial equivalence classes of the total number $N$ of equivalence classes of perfect forms  with exactly $10$ minimal vectors over the imaginary quadratic field of discriminant $\Delta$.  The fields with $\Delta \in \set{-3, -4, -7, -8}$ did not have any classes with exactly $10$ minimal vectors.} \label{tab:n10-count}
 $
\begin{array}{*{13}{r}}
  \toprule
\Delta  & {[2,6,2]} & {[3,4,3]} & {[3,5,2]} & {[4,2,4]} & {[4,3,3]} & {[5,0,5]} & {[5,1,4]} & {[5,2,3]} & {[6,0,4]} & {[6,1,3]} & {[7,0,3]} & N  \\
\midrule
%-3 & 0 & 0 & 0 & 0 & 0 & 0 & 0 & 0 & 0 & 0 & 0
% & \num{ 0 }  \\
%-4 & 0 & 0 & 0 & 0 & 0 & 0 & 0 & 0 & 0 & 0 & 0
% & \num{ 0 }  \\
%-7 & 0 & 0 & 0 & 0 & 0 & 0 & 0 & 0 & 0 & 0 & 0
% & \num{ 0 }  \\
%-8 & 0 & 0 & 0 & 0 & 0 & 0 & 0 & 0 & 0 & 0 & 0
% & \num{ 0 }  \\
-11 & 0 & 4 & 0 & 0 & 0 & 0 & 0 & 0 & 0 & 0 & 0
 & \num{ 4 }  \\
-15 & 0 & 12 & 0 & 0 & 0 & 0 & 0 & 0 & 0 & 0 & 0
 & \num{ 12 }  \\
-19 & 0 & 21 & 0 & 5 & 0 & 0 & 0 & 0 & 2 & 0 & 0
 & \num{ 28 }  \\
-20 & 5 & 28 & 0 & 0 & 0 & 0 & 4 & 6 & 0 & 0 & 0
 & \num{ 43 }  \\
-23 & 14 & 23 & 0 & 16 & 70 & 5 & 24 & 0 & 2 & 10 & 0
 & \num{ 164 }  \\
-24 & 113 & 13 & 0 & 20 & 2 & 4 & 10 & 0 & 0 & 0 & 0
 & \num{ 162 }  \\
-31 & 10 & 120 & 0 & 67 & 131 & 14 & 76 & 26 & 12 & 9 & 4
 & \num{ 469 }  \\
-35 & 11 & 153 & 0 & 43 & 104 & 26 & 34 & 12 & 12 & 2 & 0
 & \num{ 397 }  \\
-39 & 445 & 377 & 0 & 129 & 128 & 64 & 119 & 36 & 40 & 17 & 0
 & \num{ 1355 }  \\
-40 & 32 & 253 & 0 & 180 & 134 & 44 & 99 & 22 & 38 & 6 & 0
 & \num{ 808 }  \\
-43 & 29 & 110 & 0 & 95 & 156 & 47 & 101 & 12 & 38 & 20 & 0
 & \num{ 608 }  \\
-47 & 1725 & 386 & 388 & 272 & 352 & 124 & 338 & 166 & 91 & 28 & 6
 & \num{ 3876 }  \\
-51 & 786 & 536 & 0 & 155 & 170 & 80 & 154 & 35 & 45 & 28 & 2
 & \num{ 1991 }  \\
-52 & 429 & 565 & 0 & 241 & 380 & 114 & 244 & 48 & 88 & 19 & 8
 & \num{ 2136 }  \\
-55 & 50 & 914 & 0 & 517 & 826 & 260 & 433 & 136 & 156 & 41 & 12
 & \num{ 3345 }  \\
-56 & 31 & 679 & 0 & 531 & 918 & 210 & 490 & 182 & 158 & 58 & 12
 & \num{ 3269 }  \\
-59 & 49 & 521 & 0 & 357 & 518 & 188 & 446 & 124 & 180 & 29 & 10
 & \num{ 2422 }  \\
-67 & 46 & 691 & 0 & 174 & 366 & 188 & 306 & 136 & 168 & 39 & 6
 & \num{ 2120 }  \\
-68 & 180 & 941 & 0 & 721 & 1290 & 422 & 908 & 248 & 380 & 130 & 38
 & \num{ 5258 }  \\
-71 & 64 & 1587 & 216 & 1373 & 3038 & 864 & 1572 & 542 & 744 & 185 & 40
 & \num{ 10225 }  \\
-79 & 73 & 2255 & 0 & 1256 & 2366 & 872 & 1671 & 606 & 742 & 247 & 88
 & \num{ 10176 }  \\
-83 & 5640 & 1426 & 0 & 758 & 800 & 572 & 1006 & 438 & 509 & 158 & 48
 & \num{ 11355 }  \\
-84 & 3152 & 3804 & 749 & 1462 & 1068 & 710 & 1119 & 376 & 698 & 182 & 52
 & \num{ 13372 }  \\
-87 & 7643 & 4133 & 204 & 2118 & 2084 & 1198 & 2495 & 781 & 1150 & 371 & 90
 & \num{ 22267 }  \\
-88 & 137 & 2210 & 0 & 1146 & 1062 & 616 & 970 & 310 & 626 & 158 & 52
 & \num{ 7287 }  \\
-91 & 56 & 2057 & 0 & 880 & 850 & 548 & 820 & 276 & 505 & 123 & 38
 & \num{ 6153 }  \\
\midrule
\textbf{total} &  \num{20720} & \num{23819} & \num{1557} & \num{12516} & \num{16813} & \num{7170} & \num{13439} & \num{4518} & \num{6384} & \num{1860} & \num{506} & \num{109302} \\
\bottomrule
\end{array}
  $  
\end{table}
\end{landscape}
\subsection{11 minimal vectors}
Perfect forms with exactly $11$ minimal vectors are third most prevalent. Fusy~\cite{fusy} gives a formula for the number of $d$-dimensional convex polytopes with $d+3$ vertices.\footnote{There is a small error in the formula given in Theorem 1 of \cite{fusy} which we correct here.}
\begin{proposition}[{\cite[Theorem~1]{fusy}}]
  Let $c(d+3,d)$ be the number of combinatorially inequivalent $d$-dimensional polytopes with $d+3$ vertices. The generating function \[P(x)=\sum_d c(d+3,d)x^{d+3}\] has the following expression,where $\phi$ is the Euler totient function:
\begin{multline*}
P(x)=  -\frac{1}{1-x}\sum_{n \text{ odd}} \frac{\phi(n)}{4n} \ln\left(1-\frac{2x^{3n}}{(1-2x^n)^{2}}\right)\\+\frac{1}{1-x}\sum_{m\geq 1} \frac{\phi(m)}{2m}\ln\left(\frac{1-x^m}{1-2x^m}\right)   +\frac{x(x^2-x-1)(x^4-x^2+1)}{2(1-x)^2(2x^6-4x^4+4x^2-1)}\\-\frac{x(x^8-2x^7+x^6+3x^3-x^2-x+1)}{(1+x)^2(1-x)^6}.
\end{multline*}
\end{proposition}

Applying his formula, we see that there are $c(11,8) = 3210$ combinatorial types of perfect forms with exactly $11$ minimal vectors.  Of these $3210$ possible types, we only observe $964$ in the scope of our computation. Again we observe that the combinatorial types are not uniformly distributed among the types.  Some types only occur once in the range of computation, while other types were seen more frequently. Table~\ref{tab:n11-rare} lists the six combinatorial types that only occurred once in the scope of the calculation. Table~\ref{tab:n11-common} lists total number of equivalence classes of perfect forms with exactly $11$ minimal vectors as well as their distribution among the five most common combinatorial types with $11$ minimal vectors.

Since several types were exceedingly rare, it is reasonable to believe that the $964$ types is not a complete list of combinatorial types that occur for ternary perfect forms over imaginary quadratic fields.

\begin{table}
  \caption{The six combinatorial types of perfect forms with exactly $11$ vertices that were observed exactly once and the discriminant $\Delta$ of the imaginary quadratic field witnessing the type in the scope of the computation ($\abs{\Delta} \leq 91$).}
  \label{tab:n11-rare}
  $\begin{array}{cr}
    \toprule
    \text{combinatorial type} & \Delta\\
    \midrule
\ip{0, [ 0, 1, 0, 1, 0, 1, 1, 1, 0, 1, 0, 1, 0, 4 ]} & -51 \\
\ip{0, [ 0, 1, 1, 1, 1, 1, 1, 1, 0, 4 ]} &  -19 \\
\ip{0, [ 0, 2, 0, 3, 1, 1, 1, 3 ]} & -84 \\
\ip{0, [ 1, 2, 1, 2, 1, 4 ]} & -87 \\
\ip{0, [ 1, 2, 2, 2, 1, 3 ]} & -47 \\
\ip{0, [ 2, 3, 3, 3 ]} & -84 \\
\bottomrule
  \end{array}
  $
\end{table}

\begin{table}
  \caption{The distribution among the five most common combinatorial equivalence classes of the total number $N$ of equivalence classes of perfect forms  with exactly $11$ vertices over the imaginary quadratic field of discriminant $\Delta$.  The most common types are $T_1= \ip{1,[1,1,1,2,1,1,1,2]}$, $T_2 =  \ip{3,[1,1,2,1,1,2]}$, $T_3= \ip{4,[0,2,1,1,1,2]}$, $T_4 =  \ip{5,[0,1,1,0,1,1,0,2]}$, and $T_5 =\ip{5,[0,2,0,2,0,2]}$.}
  \label{tab:n11-common}
  $
\begin{array}{*{11}{r}}
\toprule
%\Delta & \ip{1,[1,1,1,2,1,1,1,2]} & \ip{3,[1,1,2,1,1,2]} & \ip{4,[0,2,1,1,1,2]} & \ip{5,[0,1,1,0,1,1,0,2]} & \ip{5,[0,2,0,2,0,2]} & N \\
\Delta & T_1 & T_2 & T_3 & T_4 & T_5 & N \\
\midrule
-3 & 0 & 0 & 0 & 0 & 0 & \num{ 0 }\\
-4 & 0 & 0 & 0 & 0 & 0 & \num{ 0 }\\
-7 & 0 & 0 & 0 & 0 & 0 & \num{ 0 }\\
-8 & 0 & 0 & 0 & 0 & 0 & \num{ 0 }\\
-11 & 0 & 0 & 0 & 0 & 0 & \num{ 0 }\\
-15 & 0 & 0 & 0 & 0 & 16 & \num{ 19 }\\
-19 & 0 & 1 & 0 & 0 & 0 & \num{ 10 }\\
-20 & 0 & 2 & 0 & 0 & 0 & \num{ 13 }\\
-23 & 0 & 0 & 0 & 0 & 8 & \num{ 33 }\\
-24 & 2 & 0 & 0 & 0 & 0 & \num{ 36 }\\
-31 & 0 & 0 & 0 & 0 & 0 & \num{ 81 }\\
-35 & 0 & 3 & 0 & 0 & 4 & \num{ 37 }\\
-39 & 1 & 4 & 28 & 0 & 31 & \num{ 251 }\\
-40 & 0 & 6 & 0 & 0 & 0 & \num{ 109 }\\
-43 & 0 & 0 & 0 & 0 & 0 & \num{ 82 }\\
-47 & 0 & 10 & 98 & 0 & 60 & \num{ 897 }\\
-51 & 6 & 25 & 0 & 0 & 31 & \num{ 297 }\\
-52 & 12 & 37 & 0 & 0 & 0 & \num{ 301 }\\
-55 & 0 & 3 & 0 & 0 & 118 & \num{ 577 }\\
-56 & 32 & 3 & 0 & 16 & 24 & \num{ 389 }\\
-59 & 0 & 1 & 0 & 0 & 52 & \num{ 359 }\\
-67 & 0 & 5 & 0 & 0 & 0 & \num{ 217 }\\
-68 & 10 & 32 & 0 & 16 & 142 & \num{ 742 }\\
-71 & 0 & 1 & 0 & 0 & 308 & \num{ 1700 }\\
-79 & 0 & 7 & 0 & 0 & 250 & \num{ 1267 }\\
-83 & 0 & 9 & 22 & 0 & 126 & \num{ 1104 }\\
-84 & 58 & 46 & 0 & 0 & 4 & \num{ 1036 }\\
-87 & 45 & 57 & 170 & 710 & 394 & \num{ 3465 }\\
-88 & 28 & 60 & 0 & 0 & 0 & \num{ 750 }\\
-91 & 0 & 6 & 0 & 0 & 456 & \num{ 893 }\\
\midrule
\textbf{total} &  \num{194} & \num{318} & \num{318} & \num{742} & \num{2024} & \num{14665} \\
\bottomrule
\end{array}
  $  
\end{table}

\subsection{12 or more minimal vectors}
Only 4886 (0.22\%) of the \num{2184773
} perfect forms in the scope of our computation had 12 or more minimal vectors.  They came in 1591 different combinatorial types.

The number of minimal vectors can be a crude measure of the complexity of a perfect form.  In Figure~\ref{fig:complexity}, we give the number of maximum number of minimal vectors occurring for each field.  From the plot, we observe that larger class number fields tend to yield less complex forms, and the complexity does not appear to grow with discriminant.

\begin{figure}
  \includegraphics[width=0.8\textwidth]{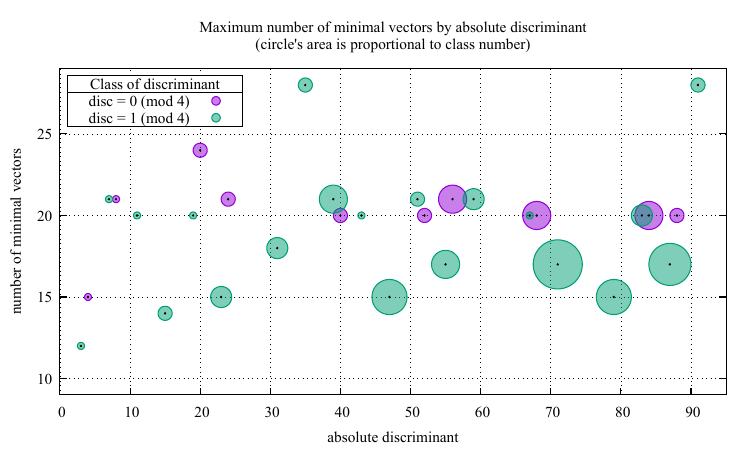}
  \caption{The maximum number of minimal vectors for a perfect form over imaginary quadratic field $F$ as a function of the absolute discriminant of $F$.  The size of the circle is proportional to the class number of $F$.  The color of the circle signifies the congruence class of the discriminant of $F$.}
  \label{fig:complexity}
\end{figure}

Now we turn our attention to the three most complex types we observed.

There is exactly one equivalence class of perfect form with 24 minimal vectors for $F = \QQ(\sqrt{-5})$.  In fact, it is the only perfect form in the scope of the calculation with 24 minimal vectors.  The Hermitian matrix \[ \renewcommand{\arraystretch}{1.3} A = \mat{
\frac{3}{2} & 
\frac{7\sqrt{-5} + 15}{20} & 
\frac{-3\sqrt{-5}}{10} \\ 
\frac{-7\sqrt{-5} + 15}{20} & 
1 & 
\frac{-\sqrt{-5}}{10} \\
\frac{3\sqrt{-5}}{10} & 
\frac{\sqrt{-5}}{10} & 
1
}\] is a representative of this class.
The $24$ minimal vectors of $A$ form an $8$-dimensional polytope.  
It has $f$-vector \[[ 1, 24, 240, 1260, 3750, 6475, 6348, 3236, 657, 1 ].\]  The $3$-dimensional faces consist of $3749$ tetrahedra and $1$ icosahedron.
The $1$-skeleton of this polytope is given in Figure~\ref{fig:p24-28ab}.
\begin{figure}
  \begin{tabular}{cc}
  \includegraphics[width = 0.35\textwidth]{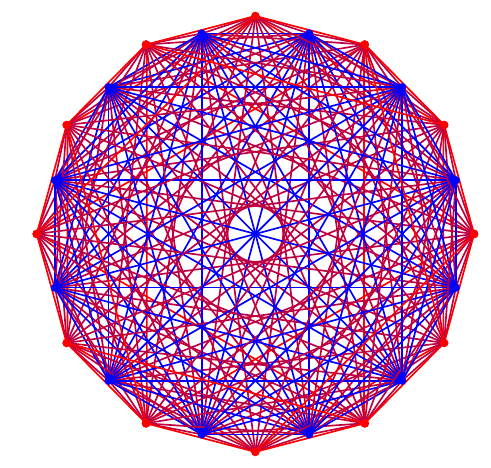} & \includegraphics[width = 0.35\textwidth]{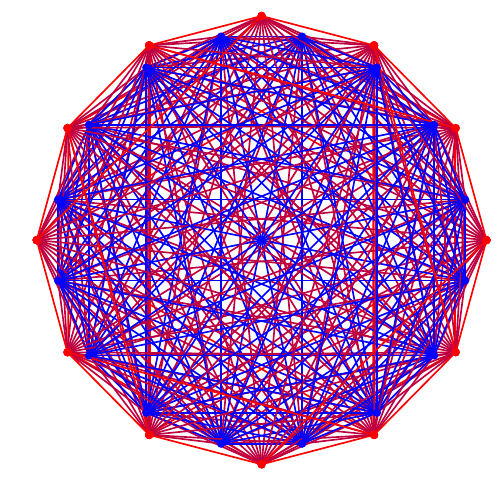}    
  \end{tabular}
  \caption{There is a unique equivalence class of perfect form over $\QQ(\sqrt{-5})$.  Its associated polytope has $24$ vertices and the $1$ skeleton is given on the left. There are two inequivalent combinatorial types with $28$ minimal vectors.  These occur for $\QQ(\sqrt{-35})$ and $\QQ(\sqrt{-91})$.  While they are combinatorially inequivalent, their associated polytopes have isomorphic $1$-skeletons.  The $1$-skeleton is given on the right.}
\label{fig:p24-28ab}  
\end{figure}

There are two combinatorial types with $28$ minimal vectors in the scope of the computation.  The first occurs in two classes of perfect forms over $\QQ(\sqrt{-35})$.
 Let $\alpha = \frac{1 + \sqrt{-35}}{2}$.  Then the two classes have representatives
\[\renewcommand{\arraystretch}{1.3}
A_1 = \mat{
1 &
\frac{-4\alpha + 12}{15} &
\frac{-2\alpha + 15}{21} \\
\frac{4\alpha + 8}{15} &
\frac{4}{3} &
\frac{6\alpha + 137}{105} \\
\frac{2\alpha + 13}{21} &
\frac{-6\alpha + 143}{105} &
\frac{7}{3}
},  
A_2 = \mat{
  \frac{5}{3} &
\frac{-10\alpha + 5}{21} &
\frac{-6\alpha - 172}{105} \\
\frac{10\alpha - 5}{21} &
\frac{4}{3} &
\frac{-53\alpha + 79}{105} \\
\frac{6\alpha - 178}{105} &
\frac{53\alpha + 26}{105} &
\frac{8}{3}
}.\]
The $28$ minimal vectors of each $A_i$ forms an $8$-dimensional polytope. These polytopes are combinatorially equivalent.  
It has $f$-vector \[[ 1, 28, 342, 2196, 7423, 13452, 13026, 6260, 1145, 1 ].\]  The $3$-dimensional faces consist of $7422$ tetrahedra and $1$ icosahedron.

The second occurs in two classes of perfect forms over $\QQ(\sqrt{-91})$.  Let $\beta = \frac{1 + \sqrt{-91}}{2}$.  Then the two classes have representatives
\[\renewcommand{\arraystretch}{1.3}
B_1 = \mat{
  \frac{6}{5} &
\frac{-9\beta + 76}{65} &
\frac{2\beta + 90}{91} \\
\frac{9\beta + 67}{65} &
2 &
\frac{8\beta + 178}{91} \\
\frac{-2\beta + 92}{91} &
\frac{-8\beta + 186}{91} &
3
}, B_2 = \mat{
\frac{6}{5} &
\frac{-73\beta + 537}{455} &
\frac{-2\beta + 92}{91} \\
\frac{73\beta + 464}{455} &
2 &
\frac{6\beta + 179}{91} \\
\frac{2\beta + 90}{91} &
\frac{-6\beta + 185}{91} &
3
}.\]

The $28$ minimal vectors of each $B_i$ forms an 8-dimensional polytope. These polytopes are combinatorially equivalent.  
It has $f$-vector \[[ 1, 28, 342, 2204, 7495, 13684, 13402, 6556, 1233, 1 ].\]  The $3$-dimensional faces consist of $7494$ tetrahedra and $1$ icosahedron.
The $1$-skeleton is isomorphic to the $1$-skeleton of $A_i$.  See Figure~\ref{fig:p24-28ab}.

\subsection{Overall}
Table~\ref{tab:totals-count} shows the number of classes of perfect forms for $\GL_3(\OO_F)$ and $\GL_2(\OO_F)$ for a range of discriminants $\Delta$.  There were $\num{2184773}$ equivalence classes distributed among the $30$ fields.   Comparing that to the 210 equivalence classes of binary perfect forms in the same range \cite{imquad-perfect}, we get a sense of the impact that the increase in rank has on the complexity of the calculation.  Table~\ref{tab:counts} shows the combinatorial equivalence classes of perfect forms, sorted by number of minimal vectors.  We see that the combinatorial types with few minimal vectors are most prevalent.  This phenomena is consistent with the computation in \cite{imquad-perfect} for binary perfect forms.

We plot the number of forms for $\GL_3(\OO_F)$ in Figure~\ref{fig:totals-count}.  A log-log plot of the data excluding the cases where the discriminant is $-3$ or $-4$ looks linear.  Indeed a power function fit excluding those cases is $N = e^{-8.8453}\abs{\Delta}^{4.8276}$ and has $R^2 = 0.9906$.  Note that the two fields that are excluded are precisely the cases where unit group is larger than $\set{\pm 1}$.

%\todo{Zach: maybe the last paragraph can expand more on the complexity of the calculation and why we bound the absolute discriminant by 91. we can say something about how long fields with absolute discriminants >100 took, and maybe how many more forms were found even in those incomplete cases}
%3526811

We conclude with a modest discuss about the extent of the calculations, specifically the bound on absolute discriminant we achieved. Due to computational limitations, calculations on fields with absolute discriminant between 100 and 200 were prematurely terminated. Indeed, calculations on 10 additional fields were initiated and found \num{3526811} polytopes. Their exclusion is due to potential bias in the data. Said differently, we have no way of knowing if the algorithm finds various polytopes at a reasonable distribution. An attempt was made to resume halted calculations, but the designed algorithm would take longer to check if previously found polytopes were rediscovered rather than compute them from scratch. Last, we could not predict how long the computer would need to run to complete fields with bigger absolute discriminant in order to justify server restart delays. For instance, assuming the reasonable power fit above, we would predict \num{6887980} polytopes in the last Euclidean field, $\mathbb{Q}(\sqrt{-163})$, where the algorithm had only found \num{493473} before it's early cessation several weeks after its initiation.

\begin{table}
  \caption{The number $N$ of ternary perfect forms and $N_2$ of binary perfect forms over the imaginary quadratic field $F$ of class number $h$ and discriminant $\Delta$.}
  \label{tab:totals-count}
  $
\begin{array}{c@{\hspace{25pt}}c}
\begin{array}[t]{*{4}{r}}
\toprule
\Delta & h & N  & N_2 \\
\midrule
-3 & 1 & \num{2} & 1 \\
-4 & 1 & \num{1} & 1 \\
-7 & 1 & \num{2} & 1 \\
-8 & 1 & \num{2} & 1 \\
-11 & 1 & \num{12} & 1 \\
-15 & 2 & \num{90} & 2 \\
-19 & 1 & \num{157} & 2 \\
-20 & 2 & \num{212} & 2 \\
-23 & 3 & \num{870} & 3 \\
-24 & 2 & \num{596} & 2 \\
-31 & 3 & \num{3953} & 4 \\
-35 & 2 & \num{3655} & 4 \\
-39 & 4 & \num{11767} & 6 \\
-40 & 2 & \num{9279} & 4 \\
-43 & 1 & \num{8457} & 4 \\
\midrule
\end{array} & 
\begin{array}[t]{*{4}{r}}
\toprule
\Delta & h & N  & N_2 \\
\midrule
-47 & 5 & \num{27538} & 9 \\
-51 & 2 & \num{19010} & 6 \\
-52 & 2 & \num{28042} & 6 \\
-55 & 4 & \num{53112} & 10 \\
-56 & 4 & \num{48768} & 9 \\
-59 & 3 & \num{44927} & 7 \\
-67 & 1 & \num{57472} & 7 \\
-68 & 4 & \num{112747} & 12 \\
-71 & 7 & \num{197938} & 16 \\
-79 & 5 & \num{258552} & 18 \\
-83 & 3 & \num{181368} & 12 \\
-84 & 4 & \num{248149} & 17 \\
-87 & 6 & \num{398124} & 17 \\
-88 & 2 & \num{255525} & 12 \\
-91 & 2 & \num{214446} & 14 \\
\midrule
\multicolumn{2}{r}{\textbf{total}} & \num{2184773} & 210\\
\bottomrule
\end{array}
\end{array}
  $  
\end{table}

\begin{figure}
  \includegraphics[width=0.8\textwidth]{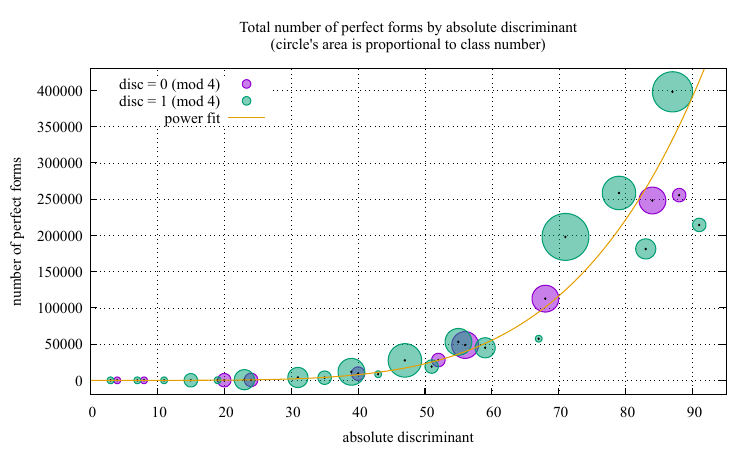}
  \caption{The number of equivalence classes of ternary perfect forms over imaginary quadratic field $F$ as a function of the absolute discriminant of $F$.  The size of the circle is proportional to the class number of $F$.  The color of the circle signifies the congruence class of the discriminant of $F$.  The curve is the power function fit
excluding the cases where the discriminant is $-3$ or $-4$.}
  \label{fig:totals-count}
\end{figure}

%\begin{figure}
%  \includegraphics[width=\textwidth]{figures/log-totals}
%  \caption{Log total number of perfect forms. \todo{expand}}  
%\end{figure}

\begin{table}
  \caption{The total number $c_n$ of combinatorial equivalence classes of perfect forms with $n$ minimal vectors observed in the ternary perfect forms over imaginary quadratic fields of absolute discriminant less than or equal to $91$, together with the number $o_n$ of times an equivalence class of perfect form with $n$ minimal vectors occurred.} \label{tab:counts}
  $
\begin{array}{c@{\hspace{25pt}}c}
  \begin{array}[t]{rrr}
\toprule
n & c_n & o_n\\
\midrule
9 & 1 & \num{2055920}\\
10 & 11 & \num{109302}\\
11 & 964 &\num{14665}\\
12 & 967 & \num{3167}\\
13 & 350 & \num{808}\\
14 & 118 &\num{477}\\
15 & 76 & \num{202}\\
16 & 28 & \num{47}\\
\midrule
  \end{array}
  &
  \begin{array}[t]{rrr}
\toprule
n & c_n & o_n\\
\midrule  
17 & 16 & \num{72}\\
18 & 9 & \num{50}\\
19 & 1 & \num{2}\\
20 & 5 & \num{26}\\
21 & 18 & \num{30}\\
24 & 1 & \num{1}\\
28 & 2 & \num{4}\\
\phantom{1}\\
\midrule
\textbf{total} & \num{2567} & \num{2184773}\\
\bottomrule
  \end{array}
\end{array}$

\end{table}

\bibliographystyle{amsalpha}
\bibliography{gl3imquad-pf.bib}

@misc{dutour9,
      title={The lattice packing problem in dimension 9 by Voronoi's algorithm}, 
      author={Dutour Sikiri\'c, Mathieu and van Woerden, Wessel},
      year={2025},
      eprint={2508.20719},
      archivePrefix={arXiv},
      primaryClass={math.NT},
      url={https://arxiv.org/abs/2508.20719}, 
}

@article {many-polytopes,
    AUTHOR = {Padrol, Arnau and Philippe, Eva and Santos, Francisco},
     TITLE = {Many regular triangulations and many polytopes},
   JOURNAL = {Math. Ann.},
  FJOURNAL = {Mathematische Annalen},
    VOLUME = {389},
      YEAR = {2024},
    NUMBER = {1},
     PAGES = {745--763},
      ISSN = {0025-5831,1432-1807},
   MRCLASS = {52B05 (52B11)},
  MRNUMBER = {4735961},
MRREVIEWER = {P.\ McMullen},
       DOI = {10.1007/s00208-023-02652-4},
       URL = {https://doi.org/10.1007/s00208-023-02652-4},
}

@article {bci-quad,
    AUTHOR = {Baeza, Ricardo and Coulangeon, Renaud and Icaza, Maria Ines
              and O'Ryan, Manuel},
     TITLE = {Hermite's constant for quadratic number fields},
   JOURNAL = {Experiment. Math.},
  FJOURNAL = {Experimental Mathematics},
    VOLUME = {10},
      YEAR = {2001},
    NUMBER = {4},
     PAGES = {543--551},
      ISSN = {1058-6458,1944-950X},
   MRCLASS = {11H55 (11R11)},
  MRNUMBER = {1881755},
MRREVIEWER = {Stefan\ K\"uhnlein},
       URL = {http://projecteuclid.org/euclid.em/1069855254},
}

@incollection {coulangeon-voronoi,
    AUTHOR = {Coulangeon, Renaud},
     TITLE = {Vorono\"i\ theory over algebraic number fields},
 BOOKTITLE = {R\'eseaux euclidiens, designs sph\'eriques et formes
              modulaires},
    SERIES = {Monogr. Enseign. Math.},
    VOLUME = {37},
     PAGES = {147--162},
 PUBLISHER = {Enseignement Math., Geneva},
      YEAR = {2001},
      ISBN = {2-940264-02-3},
   MRCLASS = {11H55 (11E12 11H50)},
  MRNUMBER = {1878749},
MRREVIEWER = {Detlev\ W.\ Hoffmann},
}

@article {watanabe-unary,
    AUTHOR = {Komatsu, Hiroyuki and Watanabe, Takao},
     TITLE = {Perfect unary forms over certain cubic fields},
   JOURNAL = {Int. J. Number Theory},
  FJOURNAL = {International Journal of Number Theory},
    VOLUME = {10},
      YEAR = {2014},
    NUMBER = {5},
     PAGES = {1337--1342},
      ISSN = {1793-0421,1793-7310},
   MRCLASS = {11E12 (11R27 11R80)},
  MRNUMBER = {3231419},
MRREVIEWER = {K.\ Szymiczek},
       DOI = {10.1142/S1793042114500304},
       URL = {https://doi.org/10.1142/S1793042114500304},
}

@incollection {watanabe-tr,
    AUTHOR = {Watanabe, Takao and Yano, Syouji and Hayashi, Takuma},
     TITLE = {Vorono\"i's reduction theory of {$GL_n$} over a totally real
              number field},
 BOOKTITLE = {Diophantine methods, lattices, and arithmetic theory of
              quadratic forms},
    SERIES = {Contemp. Math.},
    VOLUME = {587},
     PAGES = {213--232},
 PUBLISHER = {Amer. Math. Soc., Providence, RI},
      YEAR = {2013},
      ISBN = {978-0-8218-8318-1},
   MRCLASS = {11H55 (11R27 11R80)},
  MRNUMBER = {3074816},
MRREVIEWER = {Renaud\ Coulangeon},
       DOI = {10.1090/conm/587/11680},
       URL = {https://doi.org/10.1090/conm/587/11680},
}

@incollection {watanabe-survey,
    AUTHOR = {Watanabe, Takao},
     TITLE = {A survey on {V}orono\"i's theorem},
 BOOKTITLE = {Geometry and analysis of automorphic forms of several
              variables},
    SERIES = {Ser. Number Theory Appl.},
    VOLUME = {7},
     PAGES = {334--377},
 PUBLISHER = {World Sci. Publ., Hackensack, NJ},
      YEAR = {2012},
      ISBN = {978-981-4355-59-9; 981-4355-59-3},
   MRCLASS = {11H55 (20G35)},
  MRNUMBER = {2908043},
MRREVIEWER = {Renaud\ Coulangeon},
       DOI = {10.1142/9789814355605\_0011},
       URL = {https://doi.org/10.1142/9789814355605_0011},
}

@book {martinet,
    AUTHOR = {Martinet, Jacques},
     TITLE = {Perfect lattices in {E}uclidean spaces},
    SERIES = {Grundlehren der mathematischen Wissenschaften [Fundamental
              Principles of Mathematical Sciences]},
    VOLUME = {327},
 PUBLISHER = {Springer-Verlag, Berlin},
      YEAR = {2003},
     PAGES = {xxii+523},
      ISBN = {3-540-44236-7},
   MRCLASS = {11H31 (11H06 11H55 11H56)},
  MRNUMBER = {1957723},
MRREVIEWER = {Detlev\ W.\ Hoffmann},
       DOI = {10.1007/978-3-662-05167-2},
       URL = {https://doi.org/10.1007/978-3-662-05167-2},
}

@article {korkine-zolotareff,
    AUTHOR = {Korkine, A. and Zolotareff, G.},
     TITLE = {Sur les formes quadratiques positives quaternaires},
   JOURNAL = {Math. Ann.},
  FJOURNAL = {Mathematische Annalen},
    VOLUME = {5},
      YEAR = {1872},
    NUMBER = {4},
     PAGES = {581--583},
      ISSN = {0025-5831,1432-1807},
   MRCLASS = {99-04},
  MRNUMBER = {1509795},
       DOI = {10.1007/BF01442912},
       URL = {https://doi.org/10.1007/BF01442912},
}

@article {fusy,
    AUTHOR = {Fusy, \'Eric},
     TITLE = {Counting {$d$}-polytopes with {$d+3$} vertices},
   JOURNAL = {Electron. J. Combin.},
  FJOURNAL = {Electronic Journal of Combinatorics},
    VOLUME = {13},
      YEAR = {2006},
    NUMBER = {1},
     PAGES = {Research Paper 23, 25},
      ISSN = {1077-8926},
   MRCLASS = {05A15 (05A16 52B11 52B35)},
  MRNUMBER = {2212496},
MRREVIEWER = {Matthias\ Beck},
       DOI = {10.37236/1049},
       URL = {https://doi.org/10.37236/1049},
}

@article{imquad-perfect,
author = {Scheckelhoff, Kristen and Thalagoda, Kalani and Yasaki, Dan},
title = {Perfect forms over imaginary quadratic fields},
  JOURNAL = { Advanced Studies: Euro-Tbilisi Mathematical Journal special volume on Cohomology, Geometry, Explicit Number Theory},
  FJOURNAL = {Tbilisi Mathematical Journal},
note = {},
volume={9},
pages={33--46},
year={2021}
}

@article{gl3neg3,
author = {Paul E. Gunnells and Mark McConnell and Dan Yasaki},
title = {On the Cohomology of Congruence Subgroups of $\mathrm{GL}_3$ over the {E}isenstein Integers},
journal = {Experimental Mathematics},
volume = {0},
number = {0},
pages = {1-14},
year  = {2019},
publisher = {Taylor & Francis},
doi = {10.1080/10586458.2019.1577767},

URL = { 
        https://doi.org/10.1080/10586458.2019.1577767
    
},
eprint = { 
        https://doi.org/10.1080/10586458.2019.1577767
    
}

}

@article {aim-coh,
    AUTHOR = {Dutour Sikiri\'c, Mathieu and Gangl, Herbert and Gunnells,
              Paul E. and Hanke, Jonathan and Sch\"urmann, Achill and
              Yasaki, Dan},
     TITLE = {On the cohomology of linear groups over imaginary quadratic
              fields},
   JOURNAL = {J. Pure Appl. Algebra},
  FJOURNAL = {Journal of Pure and Applied Algebra},
    VOLUME = {220},
      YEAR = {2016},
    NUMBER = {7},
     PAGES = {2564--2589},
      ISSN = {0022-4049,1873-1376},
   MRCLASS = {11F75 (11F67 20G10)},
  MRNUMBER = {3457984},
MRREVIEWER = {Stefan\ K\"uhnlein},
       DOI = {10.1016/j.jpaa.2015.12.002},
       URL = {https://doi.org/10.1016/j.jpaa.2015.12.002},
}

@incollection {Yasbianchi,
    AUTHOR = {Yasaki, Dan},
     TITLE = {Hyperbolic tessellations associated to {B}ianchi groups},
 BOOKTITLE = {Algorithmic number theory},
    SERIES = {Lecture Notes in Comput. Sci.},
    VOLUME = {6197},
     PAGES = {385--396},
 PUBLISHER = {Springer},
   ADDRESS = {Berlin},
      YEAR = {2010},
   MRCLASS = {11E39 (11R34)},
  MRNUMBER = {2721434 (2012g:11069)},
       DOI = {10.1007/978-3-642-14518-6_30},
       URL = {http://dx.doi.org/10.1007/978-3-642-14518-6_30},
}

@article{Okuda2010AGO,
    AUTHOR = {Okuda, Kenji and Yano, Syouji},
     TITLE = {A generalization of {V}orono\"i's theorem to algebraic
              lattices},
   JOURNAL = {J. Th\'eor. Nombres Bordeaux},
  FJOURNAL = {Journal de Th\'eorie des Nombres de Bordeaux},
    VOLUME = {22},
      YEAR = {2010},
    NUMBER = {3},
     PAGES = {727--740},
      ISSN = {1246-7405,2118-8572},
   MRCLASS = {11H55 (11R04)},
  MRNUMBER = {2769341},
MRREVIEWER = {Renaud\ Coulangeon},
       DOI = {10.5802/jtnb.742},
       URL = {https://doi.org/10.5802/jtnb.742},
}

@article {MR86834,
	AUTHOR = {Barnes, E. S.},
	TITLE = {The perfect and extreme senary forms},
	JOURNAL = {Canadian J. Math.},
	FJOURNAL = {Canadian Journal of Mathematics. Journal Canadien de
               Math\'ematiques},
	VOLUME = {9},
	YEAR = {1957},
	PAGES = {235--242},
	ISSN = {0008-414X,1496-4279},
	MRCLASS = {10.0X},
	MRNUMBER = {86834},
	MRREVIEWER = {John\ Todd},
        DOI = {10.4153/CJM-1957-031-5},
        URL = {https://doi.org/10.4153/CJM-1957-031-5},
 }

@article {MR1580737,
    AUTHOR = {Voronoi, Georges},
     TITLE = {Nouvelles applications des param\`etres continus \`a{} la
              th\'eorie des formes quadratiques. {P}remier m\'emoire. {S}ur
              quelques propri\'et\'es des formes quadratiques positives
              parfaites},
   JOURNAL = {J. Reine Angew. Math.},
  FJOURNAL = {Journal f\"ur die Reine und Angewandte Mathematik. [Crelle's
              Journal]},
    VOLUME = {133},
      YEAR = {1908},
     PAGES = {97--102},
      ISSN = {0075-4102,1435-5345},
   MRCLASS = {99-04},
  MRNUMBER = {1580737},
       DOI = {10.1515/crll.1908.133.97},
       URL = {https://doi.org/10.1515/crll.1908.133.97},
}

@book {MR457437,
    AUTHOR = {Ash, A. and Mumford, D. and Rapoport, M. and Tai, Y.},
     TITLE = {Smooth compactification of locally symmetric varieties},
    SERIES = {Lie Groups: History, Frontiers and Applications},
    VOLUME = {Vol. IV},
 PUBLISHER = {Math Sci Press, Brookline, MA},
      YEAR = {1975},
     PAGES = {iv+335},
   MRCLASS = {14D20 (32M15 32N10)},
  MRNUMBER = {457437},
MRREVIEWER = {Ichiro\ Satake},
}

@article {MR124527,
    AUTHOR = {Koecher, Max},
     TITLE = {Beitr\"age zu einer {R}eduktionstheorie in
              {P}ositivit\"atsbereichen. {I}},
   JOURNAL = {Math. Ann.},
  FJOURNAL = {Mathematische Annalen},
    VOLUME = {141},
      YEAR = {1960},
     PAGES = {384--432},
      ISSN = {0025-5831,1432-1807},
   MRCLASS = {10.25 (32.32)},
  MRNUMBER = {124527},
MRREVIEWER = {H.\ Klingen},
       DOI = {10.1007/BF01360255},
       URL = {https://doi.org/10.1007/BF01360255},
}

@book {MR249269,
    AUTHOR = {Minkowski, Hermann},
     TITLE = {Geometrie der {Z}ahlen},
    SERIES = {Bibliotheca Mathematica Teubneriana},
    VOLUME = {Band 40},
 PUBLISHER = {Johnson Reprint Corp., New York-London},
      YEAR = {1968},
     PAGES = {vii+256},
   MRCLASS = {01.60 (10.00)},
  MRNUMBER = {249269},
}

@book {MR2466406,
    AUTHOR = {Sch\"urmann, Achill},
     TITLE = {Computational geometry of positive definite quadratic forms},
    SERIES = {University Lecture Series},
    VOLUME = {48},
      NOTE = {Polyhedral reduction theories, algorithms, and applications},
 PUBLISHER = {American Mathematical Society, Providence, RI},
      YEAR = {2009},
     PAGES = {xvi+162},
      ISBN = {978-0-8218-4735-0},
   MRCLASS = {11H55 (05B40 11J70 20B25 52-02 52B15 52B55)},
  MRNUMBER = {2466406},
MRREVIEWER = {Peter\ M.\ Gruber},
       DOI = {10.1090/ulect/048},
       URL = {https://doi.org/10.1090/ulect/048},
}

@article {MR3315516,
    AUTHOR = {Braun, Oliver and Coulangeon, Renaud},
     TITLE = {Perfect lattices over imaginary quadratic number fields},
   JOURNAL = {Math. Comp.},
  FJOURNAL = {Mathematics of Computation},
    VOLUME = {84},
      YEAR = {2015},
    NUMBER = {293},
     PAGES = {1451--1467},
      ISSN = {0025-5718,1088-6842},
   MRCLASS = {11H55 (11F06 11Y99)},
  MRNUMBER = {3315516},
MRREVIEWER = {Stefan\ K\"uhnlein},
       DOI = {10.1090/S0025-5718-2014-02891-3},
       URL = {https://doi.org/10.1090/S0025-5718-2014-02891-3},
}

@article {MR1603257,
    AUTHOR = {Franke, Jens},
     TITLE = {Harmonic analysis in weighted {$L_2$}-spaces},
   JOURNAL = {Ann. Sci. \'Ecole Norm. Sup. (4)},
  FJOURNAL = {Annales Scientifiques de l'\'Ecole Normale Sup\'erieure.
              Quatri\`eme S\'erie},
    VOLUME = {31},
      YEAR = {1998},
    NUMBER = {2},
     PAGES = {181--279},
      ISSN = {0012-9593},
   MRCLASS = {11F75 (11F70 22E41)},
  MRNUMBER = {1603257},
       DOI = {10.1016/S0012-9593(98)80015-3},
       URL = {https://doi.org/10.1016/S0012-9593(98)80015-3},
}

@book {grunbook,
    AUTHOR = {Gr{\"u}nbaum, Branko},
     TITLE = {Convex polytopes},
    SERIES = {Graduate Texts in Mathematics},
    VOLUME = {221},
   EDITION = {Second},
      NOTE = {Prepared and with a preface by Volker Kaibel, Victor Klee and
              G\"unter M.\ Ziegler},
 PUBLISHER = {Springer-Verlag, New York},
      YEAR = {2003},
     PAGES = {xvi+468},
      ISBN = {0-387-00424-6; 0-387-40409-0},
   MRCLASS = {52-01 (52Bxx)},
  MRNUMBER = {1976856},
MRREVIEWER = {Alexander\ Zvonkin},
       DOI = {10.1007/978-1-4613-0019-9},
       URL = {https://doi.org/10.1007/978-1-4613-0019-9},
}

@article {Yasunary,
    AUTHOR = {Yasaki, Dan},
     TITLE = {Perfect unary forms over real quadratic fields},
   JOURNAL = {J. Th\'eor. Nombres Bordeaux},
  FJOURNAL = {Journal de Th\'eorie des Nombres de Bordeaux},
    VOLUME = {25},
      YEAR = {2013},
    NUMBER = {3},
     PAGES = {759--775},
      ISSN = {1246-7405},
   MRCLASS = {11E12},
  MRNUMBER = {3179682},
       URL = {http://jtnb.cedram.org/item?id=JTNB_2013__25_3_759_0},
}

@article {Yascyclotomic,
    AUTHOR = {Yasaki, Dan},
     TITLE = {Binary {H}ermitian forms over a cyclotomic field},
   JOURNAL = {J. Algebra},
  FJOURNAL = {Journal of Algebra},
    VOLUME = {322},
      YEAR = {2009},
    NUMBER = {11},
     PAGES = {4132--4142},
      ISSN = {0021-8693},
     CODEN = {JALGA4},
   MRCLASS = {11E39 (11H50)},
  MRNUMBER = {2556143 (2011e:11072)},
MRREVIEWER = {Nguy{\cftil{e}}n Qu{\^o}c Th{\'a}ng},
       DOI = {10.1016/j.jalgebra.2009.06.009},
       URL = {http://dx.doi.org/10.1016/j.jalgebra.2009.06.009},
}

@article {sigrist,
    AUTHOR = {Sigrist, Fran\c cois},
     TITLE = {Cyclotomic quadratic forms},
      NOTE = {Colloque International de Th\'eorie des Nombres (Talence,
              1999)},
   JOURNAL = {J. Th\'eor. Nombres Bordeaux},
  FJOURNAL = {Journal de Th\'eorie des Nombres de Bordeaux},
    VOLUME = {12},
      YEAR = {2000},
    NUMBER = {2},
     PAGES = {519--530},
      ISSN = {1246-7405,2118-8572},
   MRCLASS = {11H55 (11E12 11H50)},
  MRNUMBER = {1823201},
MRREVIEWER = {Detlev\ W.\ Hoffmann},
       DOI = {10.5802/jtnb.295},
       URL = {https://doi.org/10.5802/jtnb.295},
}

\end{document}